\documentclass[a4paper, english,leqno,twoside,10.6pt]{amsart}
\usepackage{amssymb,amsfonts,amsthm,amsmath,latexsym, euscript}
\usepackage[T1]{fontenc}
\usepackage[utf8]{inputenc} 
\usepackage{lmodern}
\usepackage{fancyhdr}
\usepackage{babel}
\usepackage{dsfont}
\usepackage{color}
\usepackage{enumitem}
\usepackage{color}
\usepackage{ upgreek }
\usepackage{ mathrsfs }
\usepackage{mathtools}
\usepackage[normalem]{ulem} 
\definecolor{orange}{rgb}{1,0.3,0}
\definecolor{vert}{rgb}{0,0.5,0}
\definecolor{violet}{rgb}{0.7,0.1,0.8}

\numberwithin{equation}{section}
\newcommand{\R}{\mathds{R}}

\newcommand{\e}{{\rm e}}
\renewcommand{\d}{\,{\rm d}}
\newtheorem{theorem}{Theorem}[section]

\newtheorem{lemma}[theorem]{Lemma}

\renewcommand{\phi}{\varphi}
\renewcommand{\rho}{\varrho}
\renewcommand{\leq}{\leqslant}
\renewcommand{\geq}{\geqslant}
\newcommand{\gr}{\mathfrak{r}}

\newcommand{\dm} {\tfrac 12}
\renewcommand{\rhd} {\varrho_{1/2}}
\newcommand{\pnu}{p^\nu}

\newcommand{\gS}{\mathfrak{S}}

\def\e{{\rm e}}

\def\L{{\EuScript L}}

\def\gq{{\mathfrak q}}
\def\gr{{\mathfrak r}}
\makeatletter
\def\@settitle{\begin{center}%
  \baselineskip14\p@\relax
    \normalfont\LARGE
  \@title
  \end{center}%
}
\makeatother
\begin{document}
\let\MakeUppercase\relax 
	\title{ On the friable mean-value  of\goodbreak the Erd\H{o}s-Hooley Delta function}
	
	\author{\large B. Martin}
	\address{Laboratoire de Mathématiques Pures et Appliquées, CNRS, Université du Littoral Côte d’Opale,\goodbreak 50 rue F. Buisson, BP 599, Calais, 62228, France}
	\email{bruno.martin@univ-littoral.fr}

\author{\large G. Tenenbaum}
\address{Institut \'Elie Cartan, 
Universit\' e de Lorraine,
B.P. 70239, \goodbreak
F--54506 Vand\oe{}uvre-l\`es-Nancy Cedex,  
France}
\email{gerald.tenenbaum@univ-lorraine.fr}

	\author{\large J. Wetzer }
	\address{Laboratoire de Mathématiques Pures et Appliquées, CNRS, Université du Littoral Côte d’Opale, \goodbreak 
	50 rue F. Buisson, BP 599, Calais, 62228, France}
	\email{julie.wetzer@univ-littoral.fr}
	\date{\today}
\subjclass[2020]{Primary 11N25 ; Secondary 11N37}
\keywords{friable integers, Erd\H os-Hooley Delta function, mean-value of arithmetic functions, saddle-point method.}

	\maketitle

        \begin{abstract}
For integer $n$ and real $u$, define $\Delta(n,u):= |\{d : d \mid n,\,\e^u <d\leqslant  \e^{u+1} \}|$. Then, the Erd\H os-Hooley Delta function is defined as $
\Delta(n):=\max_{u\in\R} \Delta(n,u).$ 
We provide uniform upper and lower bounds for the mean-value of $\Delta(n)$ over friable integers, i.e. integers free of large prime factors.
        \end{abstract}

\section{Introduction and statement of results} 
For integer $n\geqslant  1$ and real $u$, put 
\begin{equation*}
\Delta(n,u):= |\{d : d \mid n,\,\e^u <d\leqslant  \e^{u+1} \}|, \qquad 
\Delta(n):=\max_{u\in\R} \Delta(n,u). 
\end{equation*}
The $\Delta$-function  was introduced by Erd\H{o}s in 1974 and was highlighted in 1979 by Hooley \cite{Hooley-79}. It turned out to be a key-concept in many branches of analytic number theory such as  Waring type problems, circle method, Diophantine approximation, distribution of prime factors in polynomial sequences, etc.\par 
However, the behaviour of $\Delta(n)$ remains rather mysterious. For instance,  the average order is still not known with desirable precision. Hall and Tenenbaum \cite{HaT82} obtained in $1982$ the lower bound
\begin{equation}
\label{lowbHT}
D(x) :=\sum_{n\leqslant x}\Delta(n)\gg x \log_2x \qquad(x\geqslant 3),
\end{equation}
whereas Tenenbaum \cite{T85} showed in $1985$ that  for suitable $c>0$ we have
\begin{equation}\label{eq:majo-T85}
D(x)	\ll x \e^{c\sqrt{\log_2x\log_3x}}\qquad(x\geq16).
\end{equation}
Here and in the sequel, we let $\log _k$ denote the $k$-fold iterated logarithm. Recently, La Bretèche and Tenenbaum \cite[th.\thinspace1.1]{BT22} obtained a slight improvement to \eqref{eq:majo-T85} by removing the triple logarithm in the exponent and, even more recently, Koukoulopoulos and Tao \cite{TK23} obtained the remarkable bound
$$D(x)\ll x(\log_2x)^{11/4}\qquad (x\geqslant 3).$$
A few months later, Ford, Koukouloulos and Tao \cite{FKT23}  improved \eqref{lowbHT} by showing
$$D(x)\gg  x(\log_2x)^{1+\eta +o(1)}\qquad (x\geqslant 3),$$
where the exponent $\eta\approx0.3533227$ appears in the work of Ford, Green and Koukoulopoulos \cite{FGK23}  on the normal order of $\Delta(n)$. Both bounds have been recently improved by La Bretèche and Tenenbaum \cite{BT24}: we have
\begin{equation}
\label{BTDelta}
x(\log_2x)^{3/2}\ll D(x)\ll x(\log_2x)^{5/2}\qquad (x\geqslant 3),
\end{equation}
which constitutes the current state of the art.
\medskip 

 Let $P^+(n)$ denote the largest prime factor of an integer $n>1$ and let us agree that  $P^+(1)=1$. Following  usual notation, we define $S(x,y)$ as the set of $y$-friable integers not exceeding $x$, and denote by $\Psi(x,y)$ its cardinality, viz.  
\begin{equation*}
  S(x,y):=\{n\leqslant x:P^+(n)\leqslant y \},\qquad \Psi(x,y)=|S(x,y)|\quad(x\geqslant 1,y\geqslant 1). 
\end{equation*}
Structural properties of the set $S(x,y)$ motivated a vast array of the literature in the last fourty years. The applications are indeed numerous and significant: circle method, Waring-type problems, cryptology, sieve theory, probabilistic models in number theory.
\par 
Given an arithmetical function $f$, let us use the notation $\Psi(x,y;f):=\sum_{n\in S(x,y)}f(n)$. In this work we investigate bounds for the friable mean-value  
\begin{equation}\label{defSxy}
  \gS(x,y):=\frac{\Psi(x,y;\Delta)}{\Psi(x,y)}\quad(x\geqslant y\geqslant 2). 
\end{equation}
\par \vskip-1.8mm
We now define some quantities arising in our statements. Given $\kappa>0$,  denote by 
$\rho_{\kappa}$ the continuous solution 
on $]0,\infty[$ of the delay differential system  
\begin{equation*}
  \begin{cases}
    \rho_{\kappa}(v)=v^{\kappa-1}/\Gamma(\kappa)                &\text{($0<v\leq1$)},\\
    v\rho_{\kappa}'(v)+(1-\kappa)\rho_{\kappa}(v)+\kappa\rho_{\kappa}(v-1)=0 &\text{($v>1$)},
  \end{cases}
\end{equation*}
and  set  $\rho_\kappa(v):=0$ for $v<0$.\par\goodbreak
  Thus (see, e.g., \cite{HT93}) $\varrho_\kappa$ is the order $\kappa$ fractional convolution power of $\varrho:=\varrho_1$, the Dickman function, which provides a continuous approximation to $\Psi(x,y)$ in 
\begin{equation}\label{defi:H_eps}
  	H_{\varepsilon}:=\Big\{(x,y):\ x\geq3,\ \e^{(\log_2x)^{5/3+\varepsilon}}\leqslant y\leqslant x\Big\} \qquad(\varepsilon>0).
\end{equation}
Indeed, improving on results by Dickman and de Bruijn,  Hildebrand \cite{H86} proved the asymptotic formula 
	\begin{equation}\label{esti:Psi-Hildebrand}
		\Psi(x,y)=x\rho(u)\bigg\{1+O\bigg(\frac{\log(2u)}{\log y}\bigg)\bigg\}
\quad((x,y)\in H_\varepsilon),   
	\end{equation}
with the standard notation $$u = \frac{\log x}{\log y}.$$
\par
The asymptotic behaviour of the functions $\varrho_\kappa$ (and in fact of more general delay differential equations, as displayed in \cite{HT93}) may be described in terms of the function $\xi(t)$ defined as the unique positive solution to $\e^\xi = 1+ t\xi$ for $t\neq1$ and by $\xi(1)=0$. From \cite[lemma III.5.11]{GT15} and the remark following \cite[th.\thinspace III.5.13]{GT15}, we quote the estimates 
\begin{equation}\label{eval-xi}
\xi(t)=\log t+\log_2t+O\Big({\log_2t\over \log t}\Big),\quad\xi'(t)= {1\over t}+{1\over t\log t}+O\Big({\log_2t\over t(\log t)^2}\Big)\qquad (t\to\infty).
\end{equation}
Applying \cite[cor.\thinspace2]{HT93} in the case $(a,b)=(1-\kappa,\kappa)$, we have
\begin{equation}\label{esti-int-rho_kappa}  
\rho_{\kappa}(v)= \sqrt{\frac{\xi'(v/\kappa)}{2\pi\kappa}}\exp\bigg\{\kappa\gamma-\kappa\int_{1}^{v/\kappa}\xi(t)\d 
  t\bigg\}\Big\{1+O\Big(\frac{1}{v}\Big)\Big\}\qquad (v\geqslant 1+\kappa),
\end{equation}
where $\gamma$ denotes Euler's constant. 
We put
\begin{equation}
  \label{eq:rgot}
  \gr(v):={\varrho_2(v)\over \sqrt{v}\varrho(v)}\asymp \frac1{\sqrt{v}}\exp\bigg(\int_1^v\big\{\xi(t)-\xi(t/2)\big\}\d t\bigg)\asymp 2^{v+O(v/\log 2v)}\qquad (v\geqslant 1),
\end{equation}
while a genuine asymptotic formula follows from \eqref{esti-int-rho_kappa}. 
\par
Let $\tau(n)$ denote the total number of divisors of an integer $n$. We trivially have 
\begin{equation}
\label{enctriv}
\tau(n)/\log 2n\ll \Delta(n)\leqslant \tau(n)\qquad (n\geqslant 1),
\end{equation}
 where the lower bounds follows from the pigeon-hole principle. Since, by 
 \cite[cor. 2.3]{TW03}, we have 
$$\Psi(x,y;\tau)= \bigg\{1+O\bigg(\frac{\log(2u)}{\log y}\bigg)\bigg\}x \rho_2(u) \log y \quad((x,y)\in H_\varepsilon),$$ 
we may  state as a benchmark that 
\begin{displaymath}
\frac{\gr(u)}{\sqrt{u}}  \ll \gS(x,y) \ll  2^{u+O(u/\log 2u)} \log y\qquad ((x,y)\in H_\varepsilon).   
\end{displaymath}
\par We obtain the following results, where the following notation is used: 
 \begin{align}
  &
\overline{u}:=\min\Big(\frac{y}{\log y}, u\Big)\qquad (x\geqslant y\geqslant 2),\label{def:u_ubar}\\
  &g(t):=\log \bigg({(1+2t)^{1+2t}\over (1+t)^{1+t}(4t)^{t}}\bigg) \qquad(t>0),&\label{def-g}\\ 
		&\varepsilon_y:=\frac1{\sqrt{\log y}}\quad(y\geqslant 2).&\label{def-epsy} 
\end{align}
 \begin{theorem}\label{3est}
{\rm (i)} Let $\varepsilon>0$. For a suitable absolute constant $c>0$ and uniformly for $(x,y)\in H_\varepsilon$, we have  
\begin{equation}
\label{encH-eps}\log_2 y  +  \gr(u)\ll \gS(x,y)  \ll 2^{u+O(u/\log 2u)} \e^{c\sqrt{\log_2 y \log_3 y}}.   
\end{equation}
 {\rm (ii)} For  $2\leqslant y\leqslant x^{1/(2\log_2x\log_3x)}$, and with $\lambda:=y/\log x$,  we have 
\begin{equation}\label{eval-gS-hors-Heps}
   \gS(x,y)\asymp  \e^{\{1+O(\varepsilon_y+1/\log 2u)\}g(\lambda)u}. 
\end{equation}
\end{theorem}
 Note that $g$  is positive and strictly increasing on $(0,+\infty)$. The asymptotic behaviour of this function is given by 
\begin{equation}\label{eq:asymp-g}
  g(\lambda)=
  \begin{cases}
\log 2-1/(4\lambda)+O(1/\lambda^2)  & \text{ as } \lambda\to\infty,  \\    
\lambda\log(1/\lambda)-\lambda(\log 4-1)+O(\lambda^2)  & \text{ as } \lambda\to0. 
  \end{cases}
\end{equation}
Morever, the lower bound $g(\lambda)u\gg \overline{u}$ holds on the whole range $x\geqslant y\geqslant 2$.
\par 
The error term in \eqref{eval-gS-hors-Heps} may be simplified to $1/\log 2u$ if $\log y>(\log_2x)^2$ and to $\varepsilon_y$ otherwise. 
\par 

Note that \eqref{enctriv} implies
$$\gS(x,y)\asymp\frac{\Psi(x,y;\tau)}{u^K\Psi(x,y)}\qquad \Big(\log y\leqslant \sqrt{\log x}\Big),$$
with $K=K(x,y)\in[0,2]$, so that, to the stated accuracy, the evaluation of $\gS(x,y)$ reduces in this range to that of $\Psi(x,y;\tau)/\Psi(x,y)$. This is consistent with the Gaussian tendency of the distribution of the divisors of friable integers: as the friability parameter $y$ decreases, the divisors of friable $n$ concentrate around the mean-value $\sqrt{n}$ and $\Delta(n)$ resembles more and more to $\tau(n)$, the total number of divisors. Another description of this phenomenon appears in \cite{DT18}. 
\par 
Considering available methods, Theorem \ref{3est} essentially agrees with standard expectations regarding methodology. We leave to a further project the task of adapting the method of \cite{TK23} or \cite{BT24} in the upper bound of \eqref{encH-eps}. We note right away that, in the present context, such an improvement would only be relevant for very large values of $y$ since the exponent $\sqrt{\log_2y\log_3y}$ is absorbed by the remainder $O(u/\log 2u)$ as soon as $y\leqslant x^{1/(\log_2x)^{c}}$ with $c>1/2$. 
\section{Preliminary estimates}
Here and throughout, the letter $p$ denotes a prime number.
In \cite{HT86}, Hildebrand and Tenenbaum  provided a universal estimate for $\Psi(x,y)$ by the saddle-point method. Define 
\begin{displaymath}
  \zeta(s,y):=\prod_{p\leqslant y}\Big(1-\frac{1}{p^s}\Big)^{-1},  \quad  \phi_y(s):=-\frac{\zeta'(s,y)}{\zeta(s,y)} \quad (\Re s>0 ,y\geqslant 2),
\end{displaymath}
and, for $2\leqslant y \leqslant x$, let $\alpha=\alpha(x,y)$ denote the unique positive  solution 
to the equation $\phi_y(\alpha)=\log x$. According to  \cite[th.\thinspace1]{HT86}, we have 
\begin{equation}\label{esti:Psi-Hildebrand-Tenenbaum}  \Psi(x,y)=\frac{x^\alpha\zeta(\alpha,y)}{\alpha\sqrt{2\pi|\phi_y'(\alpha)|}}\Big\{1+O\Big(\frac{1}{u}+\frac{\log y}{y}\Big)\Big\}\quad(x\geqslant y\geq2).
\end{equation}	
 By \cite[(2.4)]{HT86}), we have 
 \begin{equation}\label{alpha:esti-general}
  \alpha=\frac{\log(1+y/\log x)}{\log y}\Big\{1+O\Big(\frac{\log_2 y}{\log y}\Big)\Big\}\quad(x\geqslant y\geqslant 2).
\end{equation}
Moreover,  by \cite[(7.8)]{HT86}, we have, for any given $\varepsilon>0$,
\begin{equation}  \label{asymp-alpha-xi}
\alpha = 1 - \frac{\xi(u)}{\log y} +O\Big( \e^{- (\log y)^{(3/5)-\varepsilon}} + \frac{1}{u(\log y)^2}\Big) \quad(x\geqslant x_0(\varepsilon),\, (\log x)^{1+\varepsilon}\leqslant y\leqslant x). 
\end{equation}
Finally, by \cite[(2.5)]{HT86}, we have 
\begin{equation}
  \label{eq:eval-phi'_y}
  |\phi_y'(\alpha)| = \Big(1+ \frac{\log x}{y}\Big)\log x \, \log y \Big\{ 1+
O\Big( \frac{1}{\log(u+1)}+\frac{1}{\log y} \Big)\Big\} \quad(x\geq y\geq 2).  
\end{equation}

\medskip

\section{Proof of Theorem \ref{3est}{\rm(i)}: lower bound}
Let $\tau(n)$ denote the total number of divisors of a natural integer $n$. The following inequality is established in \cite[lemma 60.1]{HT88}
$$\Delta(n)\tau(n)\geqslant \sum_{\substack{d,d'|n \\ 0<\log (d'/d)\leqslant 1}}1=\sum_{\substack{dd'|n \\ (d,d')=1\\ 0<\log (d'/d)\leqslant 1}}\tau\Big({n\over dd'}\Big)\qquad (n\geqslant 1),$$
the equality above being obtained by representing the ratios $d'/d$ in reduced form.
\par Put
$$u_t:=\frac{\log t}{\log y}\quad (t\geqslant 1,y\geqslant 2),\quad\Omega(n):=\sum_{p^\nu\| n}\nu\quad(n\geqslant 1).$$
 Since $\tau(ab) \leqslant \tau(a) 2^{\Omega(b)} \,(a,b\geqslant 1)$, we have, for $(x,y)\in H_\varepsilon$,
\begin{equation}
  \label{eq:mingS1}
  \gS(x,y)\geqslant {1\over \Psi(x,y)}\sum_{\substack{dd'\in S(x,y) \\(d,d')=1\\ 0<\log (d'/d)\leqslant 1}} {1\over 2^{\Omega(dd')}}\Psi\Big({x\over dd'},y\Big)\gg \sum_{\substack{dd'\in S(x,y) \\ (d,d')=1 \\ 0<\log (d'/d)\leqslant 1}}{\varrho(u-u_{dd'})\over \varrho(u)dd'2^{\Omega(dd')}},
\end{equation}
where the last inequality follows from \eqref{esti:Psi-Hildebrand}.
To evaluate the double sum in \eqref{eq:mingS1}, we establish an asymptotic formula for 
$$T_d(x,y):=\sum_{\substack{m\in S(x,y)\\(m,d)=1}} \frac 1{2^{\Omega(m)}}\cdot$$
We shall make use of the following notation 
\begin{align*}
  &C:=\prod_{p}{\sqrt{1-1/p}\over 1-1/2p},\quad\kappa_y:={1\over (\log y)^{2/5}},\\
  &\varphi_y(d):=\prod_{p|d}\Big(1+{1\over 2p^{1-\kappa_y}}\Big),\quad \vartheta_y(d):=\sum_{p|d}{\log p\over p^{1-\kappa_y}}, \quad\gq(d):=\prod_{p|d}\Big(1-\frac1{2p}\Big)\quad(d\geqslant 1).\\
\end{align*}

\begin{lemma}\label{Td}
Let $\varepsilon>0$. For  $x\geq 1$, $y>\exp\{(\log_23x)^{5/3+\varepsilon}\}$,  $d\in S(x,y)$,  we have 
\begin{equation}\label{eq:Td}
T_d(x,y)={Cx\varrho_{1/2}(u)\over \sqrt{\log y}}\Big\{\gq(d)+O\Big(\kappa_y\varphi_y(d)\{1+\vartheta_y(d)\}\Big)\Big\}.
\end{equation}
\end{lemma}

\begin{proof}
We have 
$$T_d(x,y)=\sum_{m\in S(x,y)}{1\over 2^{\Omega(m)}}\sum_{t|(m,d)}\mu(t)=\sum_{t|d}{\mu(t)\over 2^{\Omega(t)}}T_1\Big({x\over t},y\Big).$$
An estimate for the inner $T_1$-term   follows from \cite[cor.\thinspace2.3]{TW03}, which, in the domain $$x\geqslant 1,\quad y>\exp\{(\log_23x)^{5/3+\varepsilon}\},$$ we rewrite as
\begin{equation}
  \label{eq:T1}
  T_1(x,y)={Cx\varrho_{1/2}(u)\over \sqrt{\log y}}\Big\{1+O\Big({\log (u+1)\over \log y}+{1\over \sqrt{\log y}}+{1\over \log(2x)}\Big)\Big\}. 
\end{equation}
Here the error term $1/\log(2x)$ enables to include the case $1\leqslant x<y$: the corresponding estimate follows from \cite[th.\thinspace II.6.2]{GT15}. 
Since $\log (u+1)\ll(\log y)^{3/5}$ in $H_\varepsilon$, we get 
\begin{equation}
  \label{evalTd1}
  T_d(x,y)={Cx\over \sqrt{\log y}}\sum_{\substack{t|d\\ t\leqslant x/\sqrt{y}}}{\mu(t)\rhd(u-u_t)\over t2^{\Omega(t)}}+R_1+R_2,
\end{equation}
with 
\begin{align*}
  R_1&\ll{x\over (\log y)^{9/10}}\sum_{\substack{t|d\\ t\leqslant x/\sqrt{y}}}{\mu(t)^2\rhd(u-u_t)\over t2^{\Omega(t)}}\ll{x\rhd(u)\over (\log y)^{9/10}}\sum_{t|d}{\mu(t)^2\over 2^{\Omega(t)}t^{1-\xi(2u)/\log y}},\\
  R_2&\ll\sum_{\substack{t|d\\ x/\sqrt{y}<t\leqslant x}}{x\over t\sqrt{\log 2x/t}},
\end{align*}
where the bound for $R_1$ follows from \begin{equation}
  \label{locrhd}
  \rhd(u-v)\ll\rhd(u)\e^{v\xi(2u)}\quad(u\geqslant 1,\,0\leqslant v\leqslant u-\tfrac12)
\end{equation}
proved in \cite{Sm93}\footnote{In \cite[lemma 6.1]{Sm93}, this bound is claimed for $0\leqslant v\leqslant u$, but it is necessary to exclude  the case when  $u-v$ is small.}. By multiplicativity, we thus get  
\begin{equation}
  \label{majR1}
  R_1\ll{x\rhd(u)\varphi_y(d)\over (\log y)^{9/10}}\cdot 
\end{equation}
\par 
Since $d\leqslant x$, we have $p_{\omega(d)} \ll \log x$, where $p_{\omega(d)}$ denotes the $\omega(d)$th prime number. Hence, using de Bruijn's estimate for $\log \Psi(x,y)$ as refined in \cite[th.\thinspace III.5.2]{GT15}, we plainly obtain, for a suitable absolute constant $c>0$,
\begin{equation}
  \label{majsomd}
  \sum_{t|d,\,t\leqslant z}1\leqslant \Psi(z,p_{\omega(d)})\leqslant z^{c/\log_2x}\qquad(\sqrt{x}\leqslant z\leqslant x\big).
\end{equation}
As a consequence 
\begin{displaymath}
  R_2\ll x\int_{\sqrt{x}}^x{1\over z}\d O\big(z^{c/\log_2x}\big)\ll \sqrt{x} \e^{c\log x/\log_2x},
\end{displaymath}
and we  conclude that
\begin{equation}
  \label{eq:R1+R2}
  R_1+R_2 \ll {x\rhd(u)\varphi_y(d)\over (\log y)^{9/10}}\cdot
\end{equation}
\par 
To estimate the main term of \eqref{evalTd1}, we approximate $\rhd(u-u_t)$ by  $\rhd(u)$, using the bound 
$$\rhd'(w)\ll\rhd(w)\log(1+w)\qquad (w\geqslant \tfrac12)$$
which, with an appropriate modification of the range of validity, is also proved in \cite[lemma~6.2]{Sm93}. In view of \eqref{locrhd}, this implies that 
$$\rhd(u-u_t)-\rhd(u)\ll u_t\rhd(u) t^{\kappa_y} \log(u+1).$$ 
Thus,  
\begin{align*}
  \sum_{\substack{t|d\\ t\leqslant x/\sqrt{y}}} {\mu(t)\rhd(u-u_t)\over t2^{\Omega(t)}\rhd(u)}&=\sum_{\substack{t|d\\ t\leqslant x/\sqrt{y}}}{\mu(t)\over t2^{\Omega(t)}}+O\Bigg( \sum_{\substack{t|d\\ t\leqslant x/\sqrt{y}}}{\mu(t)^2(\log t)\log(u+1)\over t^{1-\kappa_y}2^{\Omega(t)}\log y}\Bigg) \\
  &=\gq(d)+  O\Bigg(\sum_{\substack{t\mid d \\ x/\sqrt{y}< t\leqslant x}} \frac 1t+  \kappa_y\sum_{t|d}{\mu(t)^2\log t\over t^{1-\kappa_y}2^{\Omega(t)}}\Bigg).
\end{align*}
By \eqref{majsomd}, the first error term is $\ll \sqrt{y}x^{-1+c/\log_2x}$, which is compatible with \eqref{eq:Td}. To estimate the second, we  write $\log t= \sum_{p\mid t} \log p$ since $\mu^2(t)=1$ and invert summations. This yields the required estimate
\eqref{eq:Td}. 
\end{proof}

\medskip

By \eqref{eq:mingS1}, we have 
\begin{equation*}
\gS(x,y)\gg \sum_{d\in S(\sqrt{x}/\e,y)}\frac{\varrho(u-2u_d)\{T_d(\e d,y)-T_d(d,y)\}}{\varrho(u)d^22^{\Omega(d)}}\cdot
\end{equation*}
We insert \eqref{eq:Td} to evaluate the difference between curly brackets and sum separately the resulting main term and the remainder terms. 
This can be done  by partial summation, using a variant of~\eqref{eq:T1} in which the inclusion of the factors $\gq(d)$ or $\varphi_y(d)\{1+\vartheta_y(d)\}$ has as sole effects to alter the value of the constant $C$.
This yields
\begin{equation}
  \label{eq:mino-gS-integrale}
  \begin{aligned}
  \gS(x,y)&\gg \sum_{d\in S(\sqrt{x}/\e,y)}{\gq(d)\varrho(u-2u_d)\rhd(u_d)\over \varrho(u)2^{\Omega(d)}d\sqrt{\log y}}\\
&\gg\frac1{\varrho(u)}\int_{1/\log y}^{u/2}{\varrho(u-2v)\rhd(v)^2}\d v=\frac1{2\varrho(u)}\int_{2/\log y}^u\varrho(u-v)\rhd(\dm v)^2\d v.
 \end{aligned}
\end{equation}
The contribution of the interval $[2/\log y,2]$ to the last integral  is 
\begin{equation}
\label{eq:int-voisnage-0}
  \geqslant 2\rho(u)\int_{1/\log y}^1 \rho_{1/2}(v)^2 \d v = \frac{2\rho(u)}{\pi} \int_{1/\log y}^1 \frac{\d v}{v}
 =  \frac{2\rho(u)}{\pi} \log_2y. 
\end{equation}
Now observe that \eqref{esti-int-rho_kappa} implies
$$\rhd(\dm v)^2\asymp\frac{\varrho(v)}{\sqrt{v}}\qquad (v\geqslant 1).$$
Since $\rho_2$ is the convolution square of $\varrho$, it follows that
\begin{equation}
\label{eq:int-final}
\frac1{\varrho(u)}\int_{2}^u\varrho(u-v)\rhd(\dm v)^2\d v\gg\frac{\varrho_2(u)}{\sqrt{u}\varrho(u)}=\gr(u).
\end{equation}

Carrying back into \eqref{eq:mino-gS-integrale} and taking \eqref{eq:int-voisnage-0} into account, we obtain the required estimate.
\medskip
\section{Proof of Theorem \ref{3est}{\rm(i)}: upper bound}
We adapt to the friable case the iterative method developed by Tenenbaum in \cite{T85} (see also \cite[\S 7.4]{HT88}) for bounding the mean-value of the $\Delta$-function.
Throughout this proof the letters $c$ and $C$, with or without index,  stand for absolute positive constants.\par 
Given an integer $n\geqslant 2$, let us denote by $\{p_j(n)\}_{1\leqslant j\leqslant \omega(n)}$ the increasing sequence of its distinct prime factors. 
Following \cite{T85} (see also \cite{HT88}), define $$M_q(n)=\int_\R \Delta(n,u)^q\d u,$$
and, for squarefree $n$, put  $$n_k:=\begin{cases}\displaystyle\prod_{j\leqslant k}p_j(n) & \text{ if } k\leqslant \omega(n),\\
n& \text{ otherwise.}
\end{cases}$$ 
 Now,  let $$L_{k,q}=L_{k,q}(x,y):=\sum_{P^+(n)\leqslant y}{\mu(n)^2M_q(n_k)^{1/q}\over n^\beta},$$
where $\beta:=\alpha\big(\sqrt{x},y\big)$ is the saddle-point related to the friable mean-value of $\tau(n)$, the divisor function.
\par 
We aim at bounding $L_{k,q}$ from above for large $k$ and $q$. The starting point is the identity $$\Delta(mp,u)=\Delta(m,u)+\Delta(m,u-\log p)\quad(u\in\R,\,p\nmid m).$$ 
Apply this to $m=n_k$, $p=p_{k+1}(n)$. Raising to the power $q$ and expanding out, we obtain  $$M_q(n_{k+1})=2M_q(n_k)+E_q(n_k,p_{k+1})\qquad (\omega(n)>k),$$
with
  \begin{displaymath}
    E_q(m,p) := \sum_{1\leqslant j<q} \binom{q}{j} \int_{\R} \Delta(m;v)^j
 \Delta(m;v-\log p)^{q-j} \d v. 
  \end{displaymath}

It follows that 
 $$L_{k+1,q}\leqslant 2^{1/q}L_{k,q}+\sum_{\substack{P^+(m)\leqslant y\\ \omega(m)=k}}\mu(m)^2\sum_{P^+(m)<p\leqslant y}E_q(m,p)^{1/q}\sum_{\substack{P^+(n)\leqslant y\\ \omega(n)\geqslant k+1\\ n_{k+1}=mp}}{\mu(n)^2\over n^\beta}\cdot$$
The latter sum is $$\ll {\zeta_1(\beta,y)\over p^\beta m^\beta}\prod_{\ell\leqslant p}{1\over 1+\ell^{-\beta}}=:{\zeta_1(\beta,y)g_\beta(p)\over p^\beta m^\beta},$$
where, here and in the remainder of this proof, $\ell$ denotes a prime number, and $$\zeta_1(\sigma,y):=\prod_{\ell\leqslant y}(1+\ell^{-\sigma}).$$
 Hölder's inequality yields
 $$\sum_{z<p\leqslant y}{E_q(m,p)^{1/q}\over p^\beta}\leqslant \bigg\{\sum_{p\geqslant 2}{E_q(m,p)\log p\over p}\bigg\}^{1/q}\bigg\{\sum_{z<p\leqslant y}{1\over p^{(q\beta-1)/(q-1)}(\log p)^{1/(q-1)}}\bigg\}^{(q-1)/q}.$$
and the prime number theorem enables to bound the last sum over $p$ by
 $$\ll {qy^{q(1-\beta)/(q-1)}\over (\log z)^{1/(q-1)}}.$$
Now, we have (see, e.g.,  \cite[th. 73]{HT88}) 
$$\sum_{p}{E_q(m,p)\log p\over p}\leqslant C4^q\tau(m)^{q/(q-1)}M_q(m)^{(q-2)/(q-1)}.$$
It follows that 
 \begin{equation}
   \label{majrec}
   L_{k+1,q}
\leqslant 2^{1/q}L_{k,q}+C_1q\e^{\xi(u/2)}G_k\leqslant 2^{1/q}L_{k,q}+C_2qu^2G_k,
 \end{equation}
with 
$$G_k:=\zeta_1(\beta,y)\sum_{\substack{P^+(m)\leqslant y\\ \omega(m)=k}}{\mu(m)^2\tau(m)^{1/(q-1)}M_q(m)^{(q-2)/q(q-1)}g_\beta(P^+(m))\over m^\beta(\log P^+(m))^{1/q}}\cdot$$
Since 
$${\mu(m)^2\zeta_1(\beta,y)g_\beta(P^+(m))\over m^\beta}= \sum_{\substack{P^+(n)\leqslant y\\ n_k=m}}{\mu(n)^2\over n^\beta},$$
we infer that 
$$G_k\leqslant \sum_{\substack{P^+(n)\leqslant y \\ \omega(n)\geqslant k}}{\mu(n)^2\tau(n_k)^{1/(q-1)}M_q(n_k)^{(q-2)/q(q-1)}\over n^\beta(\log p_k(n))^{1/q}}\cdot$$
A new application of Hölder's inequality yields
$$G_k\leqslant L_{k,q}^{(q-2)/(q-1)}S_k^{1/(q-1)},$$
 where 
 \begin{align*}
   S_k&:=\sum_{\substack{P^+(n)\leqslant y\\ \omega(n)\geqslant k}}{\mu(n)^2\tau(n_k)\over n^\beta\{\log p_k(n)\}^{(q-1)/q}}\\
&\leqslant 2\sum_{\substack{P^+(m)\leqslant y\\ \omega(m)=k-1}}{\mu(m)^2\tau(m)\over m^\beta}\sum_{P^+(m)<p\leqslant y}{1\over p^\beta(\log p)^{1-1/q}}\prod_{p<\ell\leqslant y}\Big(1+{1\over \ell^\beta}\Big)\\
&\leqslant {\zeta_{1}(\beta,y)\over (k-1)!}\sum_{p\leqslant y}{{g_\beta(p)\over p^\beta(\log p)^{1-1/q}}}\Big(\sum_{\ell\leqslant p}{2\over \ell^\beta}\Big)^{k-1}\ll{\zeta_{1}(\beta,y)y^{1-\beta}\over (k-1)!}\sum_{p\leqslant y}{{\e^{-T(p)}\{2T(p)\}^{k-1}\over p(\log p)^{1-1/q}}},
 \end{align*}
where we set $$
T(p):=\sum_{\ell\leqslant p}{1\over \ell^\beta}\cdot$$
(Recall that the letter $\ell$ denotes generically a prime number.) \par 
We evaluate $T(p)$ by \cite[lemma 3.6]{BT05a}. Writing
$$\L(z):=\e^{(\log z)^{3/5}/(\log_2z)^{1/5}},\quad w(t):={t^{1-\beta}-1\over (1-\beta)\log t},$$ we have 
$$T(p)=\log_2p+\int_1^{w(p)}t\xi'(t)\d t+b+O\Big(\frac{w(p)}{\L(p)^c} + \frac{\log (u+1)}{\log y}\Big).$$
where $b$ is a suitable constant.
Note that $w(y)=u/2+O(u/\log y)$. Defining $$h(v):= \int_1^{w(\exp\e^v)}t\xi'(t)\d t+b_1,$$ with $b_1$ sufficiently large so that  $T(p) \leqslant \log_2 p +h(\log_2 p)$, and writing $z_v:=v+h(v)$,
we have, by the prime number theorem,
$$W_k(y):=\sum_{p\leqslant y}{{\e^{-T(p)}\{T(p)\}^{k-1}\over p(\log p)^{1-1/q}}}\ll \int_0^{\log_2y}\e^{-(2-1/q)z_v+(1-1/q)h(\log_2y)}z_v^{k-1}\d v.$$
Since $h(\log_2y)\leqslant u/2+O(u/\log 2u)$ and since $h'(v)\geqslant 0$, the change of variables $z=z_v$ yields
$$W_k(y)\ll \e^{u/2+O(u/\log 2u)}\int_0^\infty\e^{-(2-1/q)z}z^{k-1}\d z\ll \frac{\e^{u/2+O(u/\log 2u)}(k-1)!}{(2-1/q)^{k-1}}\cdot$$
\par 
Thus, 
$$S_k\ll{\zeta_{1}(\beta,y)\e^{u/2+O(u/\log 2u)}\over (1-1/2q)^k}\ll{\zeta_1(\alpha,y)\e^{O(u/\log 2u)}\over (1-1/2q)^k},$$ 
since $\zeta(\beta,y)=\zeta(\alpha,y)\e^{-u/2+O(u/\log 2u)}$ --- see \cite[(4.2)]{Te22}.\par 
 Finally, for $q$ sufficiently large and $\dm<\lambda<\log 2$, we obtain 
\begin{equation}
  \label{recLkq}
  G_k\leqslant C_{ 3} L_{k,q}^{(q-2)/(q-1)}\zeta_{1}(\alpha,y)^{1/(q-1)}\e^{c_0u/(q\log 2u)+\lambda k/q(q-1)}.
\end{equation}
At this stage, we introduce  $$L^*_{k,q}=L_{k,q}+2^{k/q}u^{2q}\e^{c_0u/\log 2u}\zeta_1(\alpha,y),$$ so that  \eqref{majrec} still holds for  $L^*_{k,q}$ in place of $L_{k,q}$. Setting $q(k):=\big\lfloor c_1\sqrt{k/\log k}\big\rfloor$ with sufficiently small, absolute $c_1$, we thus have, for large $k$,  
$$L^*_{k+1,q}\leqslant \Big\{2^{1/q}+{1\over k}\Big\}L_{k,q}^*\qquad \big(q\leqslant q(k)\big),$$
whence 
\begin{equation}
  \label{reck}
  L^*_{k+1,q}\leqslant 3^{1/q} L^*_{k,q}\qquad \big(q\leqslant q(k)\big).
\end{equation}
\par 
To carry out a double induction on $k$ and $q$, we also need a bound on  $L_{k,q+1}^*$ in terms of $L_{k,q}^*.$ This is achieved by the inequality $M_{q+1}(n)^{1/(q+1)}\leqslant 2M_q(n)^{1/q}$ proved in \cite[th. 72]{HT88}, which yields 
\begin{equation}
  \label{recq}
  L_{k,q+1}^*\leqslant 2u^2L_{k,q}^*.
\end{equation}
 With the aim of bounding  $L_{k,q(k)}^*$ in terms of $L_{2,q(2)}^*$, we use  \eqref{reck} to reduce the parameter $k$ and \eqref{recq} to secure the condition $q\leqslant q(k)$. The first handling provides an overall factor 
$$\leqslant \prod_{1\leqslant q\leqslant q(k)}q^{c_2}\leqslant \e^{c_3\sqrt{k\log k}}$$
whereas the second induces a global factor 
$\ll u^{c_4q(k)}.$
\par 
Finally,  we obtain
$$L_{k,q}^*\ll L_{2,q(2)}^*u^{c_5q(k)}\e^{c_5\sqrt{k\log k}}.$$
Let $K:=\log_2y+u.$ It can be shown (see \cite{BT05b} and use a bound similar to \cite[(7.44)]{HT88}) that the contribution to  $L_{k,q}$  of those integers  $n$ such that  $\omega(n)>CK$ is negligible, and we omit the details. Eventually, we arrive at 
$$L_{k,q}\ll\e^{c_5\sqrt{K\log K}}u^{c_6\sqrt{K/\log K}}\zeta(\alpha,y)\e^{c_0u/\log 2u}\ll\zeta(\alpha,y)\e^{c_7\sqrt{\log_2y\log_3y}+O(u/\log 2u)},$$
and so
$$\sum_{n\in S(x,y)}{\mu(n)^2\Delta(n)\over n^\beta}\ll \zeta(\alpha,y)\e^{c\sqrt{(\log_2y)\log_3y}+O(u/\log 2u)}.$$
Employing the representation $n=mr^2$, $\mu(m)^2=1$, we obtain that the same bound holds for 
$$\sum_{n\in S(x,y)}{\Delta(n)\over n^\beta}\cdot$$
\par 
This is the key to our upper bound for $D(x,y):=\sum_{n\in S(x,y)}\Delta(n)$. We have 
\begin{align*}
  D(x,y)\log x-\int_1^x{D(t,y)\over t}\d t&=\sum_{n\in S(x,y)}\Delta(n)\log n\leqslant \sum_{\substack{m\pnu\leqslant x\\ P^+(mp)\leqslant y}}\Delta(m)(\nu+1)\log \pnu\\
        &\ll yD\Big({x\over y},y\Big)+\sum_{\substack{x/y<n\leqslant x\\P^+(n)\leqslant y}}{x\Delta(n)\over n}+\sum_{\substack{n\leqslant x \\ P^+(n)\leqslant y}}\Delta(n)\sqrt{x\over n}\cdot
\end{align*}
The trivial bound
$$D(x,y) \leqslant \sum_{n\in S(x,y)} \tau(n) \ll x\varrho_2(u)\log y,$$
that holds in $H_\varepsilon$ (see \cite[Cor. 2.3]{TW03}), 
furnishes
$$\int_1^x{D(t,y)\over t}\d t\ll x\varrho_2(u)\log y,\quad yD(x/y,y)\ll x\varrho_2(u-1)\log y.$$
Moreover, in the same region, for $y$ sufficiently large, $\beta>1/2$ 
$$\sum_{\substack{n\leqslant x\\ P^+(n)\leqslant y}}\Delta(n)\sqrt{x\over n}+\sum_{\substack{x/y<n\leqslant x\\ P^+(n)\leqslant y}}{x\Delta(n)\over n}\ll x^{\beta}\e^{\xi(u/2)}\sum_{n\in S(x,y)}{\Delta(n)\over n^\beta}\cdot$$
Collecting these estimates, we obtain   
\begin{align*}
  D(x,y)&\ll x{\varrho_2(u)\over u}+x\varrho_2(u)\log 2u+{x^\beta \zeta(\alpha,y)\e^{c\sqrt{(\log_2y)\log_3y}+O(u/\log 2u)}\over \log x}
\\ 
& \ll \Psi(x,y) 2^{u+O(u/\log 2u)}\e^{c\sqrt{\log_2 y\log_3y}},
\end{align*}  
where we used \eqref{eq:rgot}, \eqref{esti:Psi-Hildebrand},   the estimate  
\begin{displaymath}
  \frac{x^\beta \zeta(\alpha,y)}{ \log x} \asymp   \Psi(x,y) 2^{u+O(u/\log u)},
\end{displaymath}
which follows from \eqref{esti:Psi-Hildebrand-Tenenbaum}, \eqref{eq:eval-phi'_y} and 
\begin{displaymath}
  (\beta-\alpha) \log x  = - u \int_{u/2}^u \xi'(t)dt  +O(1) = u\log 2 +O\Big(\frac{u}{\log u}\Big). 
\end{displaymath}

 This concludes the proof of the upper bound included in \eqref{encH-eps}.
\medskip

\section{Proof of Theorem \ref{3est}{\rm(ii)}}\label{sec:proof-thm-ii}
We retain notation $g(t)$ from \eqref{def-g}, $\varepsilon_y$ from \eqref{def-epsy}, define
$ \eta_y:= (\log_2 y)/\log y.$ 
Since  $\max(1,\lfloor \tau(n)/\log n\rfloor) \leqslant \Delta(n) \leqslant \tau(n)$ holds for all $n\geqslant 1$ (see e.g. \cite[th. 60, (6.7)]{HT88}), we have 
\begin{equation}
\label{encgS}
\frac{\Psi(x,y;\tau)}{2\Psi(x,y)\log x }\leqslant \gS(x,y) \leqslant \frac{\Psi(x,y;\tau)}{\Psi(x,y)}\qquad (x\geqslant y\geqslant 2)\cdot
\end{equation}
Now, by \cite[th.\thinspace1.2]{Te22} and \cite[(1.6)]{Te22},  we have, with $\lambda:=y/\log x$, 
 \begin{equation}
   \label{evalgS}
\frac{\Psi(x,y;\tau)}{\Psi(x,y)}
    \asymp \zeta(\alpha,y)\e^{-uh(\lambda)\{1+O(\varepsilon_y)\}} \quad(x\geqslant y\geqslant 2),
 \end{equation}
 where we have put
 \begin{align*}
  h(t)&:=t\log 4-(1+2t)\log \Big(\frac{1+2t}{ 1+t}\Big)=t\log \Big(1+\frac1t\Big)-g(t)\quad (t\geqslant 0),
\end{align*}
and, for the purpose of further reference, note that 
\begin{equation}
  \label{evaltheta}
 uh(\lambda)\sim(1-\log 2)u\quad(u\to\infty,\,\lambda\to\infty),\quad uh(\lambda)\asymp \overline{u} \qquad (x\geqslant y\geqslant 2).
\end{equation}

We shall show that 
\begin{equation}
  \label{evalZay}
  \zeta(\alpha,y)=\e^{\lambda u\log (1+1/\lambda)\{1+O(\varepsilon_y+1/\log 2u)\}}\qquad\Big(2\leqslant y\leqslant x^{1/(2\log_2x \log_3x)}\Big).
\end{equation}
\par 

 Since $g(\lambda)u\gg \overline u$ for $x\geqslant y\geqslant 2$, we see that  \eqref{eval-gS-hors-Heps} follows from \eqref{encgS} and \eqref{evalZay} in any subregion where $\overline u (\varepsilon_y+1/\log 2u)\gg\log_2x$: the condition above corresponds to this requirement when $y$ is large. However, for bounded $y$, we have $\Psi(x,y;\tau)/\Psi(x,y)\asymp(\log x)^{\pi(y)}$, and so \eqref{eval-gS-hors-Heps} holds trivially. Therefore, we may assume in the sequel that $y$ is sufficiently large. 

\par  

Let us now embark on the proof of \eqref{evalZay}.  
\par
Observe that 
\begin{equation}
\label{zay}
\zeta(\alpha,y)=\zeta(1,y)\exp\bigg\{\int_{\alpha}^{1} \varphi_y (\sigma) \d \sigma  \bigg\}.
\end{equation}
Using the estimate for $\varphi_y(\sigma)$ given in \cite[lemma 13]{HT86}, we may write 
\begin{equation}
  \label{Z1}
  \int_{\alpha}^1\varphi_y(\sigma)\d \sigma=\Big\{1+O\Big({1\over \log y}\Big)\Big\}\int_\alpha^1{y^{1-\sigma}-1\over (1-\sigma)(1-y^{-\sigma})}\d\sigma.
\end{equation}
\par 
By inspection of the proof  of \eqref{alpha:esti-general} in \cite[pp. 285-7]{HT86}, we see that, for a suitable constant $C$, we have
\begin{equation}
  \label{alpha+}
\alpha(x,y)=1-{\xi(u)\over \log y}+O\Big({1\over (\log y)^2}\Big)\qquad \Big(C(\log x)(\log_2x)^3< y\leqslant x\Big).  
\end{equation}
This implies $y^\alpha\gg y\e^{-\xi(u)}\gg\log y$ in the same domain, so the contribution of the term $1-y^{-\sigma}$ in \eqref{Z1} is absorbed by the error term. The change of variables defined by  \mbox{$(1-\sigma)\log y=\xi(t)$} then provides, in view of \eqref{alpha+} and  \eqref{eval-xi}, 
\begin{equation*}
\begin{aligned}
\int_{\alpha}^1\varphi_y(\sigma)\d \sigma&=\Big\{1+O\Big({1\over \log y}\Big)\Big\}\int_1^ut\xi'(t)\d t\\
&=u+{u\over \log u}+O\Big({u\over \log y}\Big)=u+O\Big(\frac u{\log 2u}+\varepsilon_y u\Big).
\end{aligned}
\end{equation*}
Since, in the domain of \eqref{alpha+},
$$u\lambda\log \Big(1+\frac1\lambda\Big)=u+O\Big({u\over \log 2u}\Big),
$$
we obtain \eqref{evalZay} in the range $C(\log x)(\log_2x)^3< y\leqslant x^{1/(2\log_2x\log_3x)}$. Indeed the factor $\zeta(1,y)\asymp \log y$ appearing in \eqref{zay} is absorbed by the error term.
\par 
When $2\leqslant y\leqslant C(\log x)(\log_2x)^3$,  we put $t=y^\sigma$ in \eqref{Z1} to get
$$\int_{\alpha}^1\varphi_y(\sigma)\d \sigma=\Big\{1+O\Big({1\over \log y}\Big)\Big\}\int_{y^\alpha}^y{y/t-1\over (t-1)\log (y/t)}\d t.$$
Note that \eqref{alpha:esti-general} now implies $\alpha\log y\ll\log_22y $. Put $T:=(\log y)^K$, where $K$ is so large so that $T>y^{\alpha}$. The contribution of the interval $[T,y]$ to the above integral is 
$$\ll \int_T^\infty\frac{y}{t^2}\d t\ll  \frac{y}{(\log y)^K},$$
where we used the bound $\e^v-1\ll v \e^v$ $(v\geqslant 0)$. Then, 
\begin{align*}
  \int_{y^\alpha}^T{y/t-1\over (t-1)\log (y/t)}\d t
  &=\Big\{1+O\Big({\log_2y\over \log y}\Big)\Big\}{y\over \log y}\int_{y^\alpha}^T{1\over t(t-1)}\d t\\
  &=\big\{1+O\big(\eta_y\big)\big\}{y\over \log y}\log \Big({1-1/T\over 1-1/y^\alpha}\Big).
\end{align*}
>From \cite[(7.18)]{HT86}, it follows that
$$			\log\Big(\frac{1}{1-y^{-\alpha}}\Big)=\log\Big(1+\frac1\lambda\Big)\Big\{1+O\Big(\frac{\log_2y}{\log y}\Big)\Big\}\quad\Big(2\leqslant y\leqslant C(\log x)(\log_2x)^3\Big).
$$
Therefore
\begin{align*}
  \int_{\alpha}^1\varphi_y(\sigma)\d \sigma &= 
\big\{1+O\big(\eta_y\big)\big\}{y\over \log y}\log \Big(1+\frac1\lambda\Big)+O\Big(\frac y{(\log y)^K}\Big)\\
 &=\big\{1+O\big(\eta_y\big)\big\}u\lambda\log \Big(1+\frac1\lambda\Big)\cdot
\end{align*}
This establishes \eqref{evalZay} in the complementary range $2\leqslant y\leqslant C(\log x)(\log_2x)^3$.
\par  This completes the proof of theorem~\ref{3est}(ii).\par 

\section*{Acknowledgement}

This work is supported by the Austrian-French project “Arithmetic Randomness” between FWF and ANR (grant numbers I4945-N and ANR-20-CE91-0006).\par

\vskip2cm

\bibliographystyle{plain}
\bibliography{biblio-hooley}

\end{document}